\documentclass{article}

\usepackage{amsfonts}
\usepackage{amsmath}
\usepackage{amsrefs}
\usepackage[all]{xy}
\usepackage{amssymb}
\usepackage{amsthm}
\usepackage[english]{babel}
\usepackage{color}
\usepackage{epsfig}
\usepackage{graphicx}
\usepackage{hyperref}
\usepackage[utf8]{inputenc}
\usepackage{mathrsfs}
\usepackage{nicefrac}
\usepackage{psfrag}
\usepackage{xcolor}
\usepackage{comment}
\usepackage{tikz}
\usepackage{float}
\usepackage[T1]{fontenc}
\usetikzlibrary{decorations.markings}
\usetikzlibrary{graphs}
\usepackage{setspace}
\usepackage[margin=2cm]{geometry}
\hypersetup{
 colorlinks,
 citecolor=blue,
 filecolor=blue,
 linkcolor=blue,
 urlcolor=blue}

\newcommand{\R}{\mathbb R}
\newcommand{\Z}{\mathbb Z}
\newcommand{\N}{\mathbb N}

\providecommand{\abs}[1]{\lvert#1\rvert}

\DeclareMathOperator{\aster}{\mathbin{\raisebox{-0.2ex}{${}^*\!$}}}

\DeclareMathOperator{\dist}{dist}
\DeclareMathOperator{\iti}{it}
\DeclareMathOperator{\st}{st}

\theoremstyle{plain}
\newtheorem{theorem}{Theorem}[section]

\newtheorem{lemma}[theorem]{Lemma}
\newtheorem{proposition}[theorem]{Proposition}

\theoremstyle{definition}

\newtheorem{remark}[theorem]{Remark}
\newtheorem{example}[theorem]{Example}

\theoremstyle{remark}

\definecolor{ao}{rgb}{0.0, 0.5, 0.0}

\makeatletter
\newsavebox\myboxA
\newsavebox\myboxB
\newlength\mylenA
\newcommand*\xoverline[2][0.75]{%
    \sbox{\myboxA}{$\m@th#2$}%
    \setbox\myboxB\null
    \ht\myboxB=\ht\myboxA%
    \dp\myboxB=\dp\myboxA%
    \wd\myboxB=#1\wd\myboxA
    \sbox\myboxB{$\m@th\overline{\copy\myboxB}$}
    \setlength\mylenA{\the\wd\myboxA}
    \addtolength\mylenA{-\the\wd\myboxB}%
    \ifdim\wd\myboxB<\wd\myboxA%
       \rlap{\hskip 0.5\mylenA\usebox\myboxB}{\usebox\myboxA}%
    \else
        \hskip -0.5\mylenA\rlap{\usebox\myboxA}{\hskip 0.5\mylenA\usebox\myboxB}%
    \fi}
\makeatother

\begin{document}

\title{Nonstandard analysis of asymptotic points of expansive systems}
\author{Alfonso Artigue, Luis Ferrari and Jorge Groisman}

\maketitle
\begin{abstract} 
In this paper we apply techniques from nonstandard analysis to study expansive dynamical systems.
Among other results, we provide a necessary and sufficient condition for an expansive homeomorphism on a compact metric space to admit doubly-asymptotic points in terms of the decay of expansivity constants of the powers of the system.
\end{abstract}

\section{Introduction}
\label{secIntro}
Let \((X,\dist)\) be a compact metric space. A homeomorphism \(f\colon X\to X\) is \textit{expansive} if there exists a constant \(c>0\) such that for all \(x,y\in X\), if \(x\neq y\), then \(\dist (f^n(x),f^n(y))>c\) for some \(n\in \mathbb{Z}\). We say that \(x\) and \(y\) are (\textit{doubly}-)\textit{asymptotic} 
if \(\dist (f^n(x),f^n(y))\rightarrow 0\) as \(n\rightarrow +\infty\) (\(n\rightarrow \pm \infty\)). 
It is known \cites{utz1950unstable,Schw} that every expansive homeomorphism on a compact metric space with infinitely many points has asymptotic points. 

Expansivity is a key property of hyperbolic sets (as Anosov diffeomorphisms, the non-wandering set of Smale's Axiom A systems and subshifts of finite type) and it is well known that such systems exhibit several doubly-asymptotic points, for instance, a homoclinic point with its limit periodic orbit. Additionally, expansive homeomorphisms on compact surfaces, which are known to be conjugate to Pseudo-Anosov homeomorphisms, also present doubly-asymptotic points. Moreover, all known examples of expansive homeomorphisms on compact metric spaces that are not totally disconnected have doubly-asymptotic points. In a totally disconnected space, there exist examples without doubly-asymptotic points, see \cites{Yo,King}. It is known that an expansive homeomorphism of a totally disconnected space is conjugate to a subshift.

For an expansive homeomorphism we can define the finite number
\begin{equation}
\label{ecuGamma}
\gamma(f)=\sup\{c > 0: c \text{ is an expansivity constant of } f\}.
\end{equation}
It is known and easy to prove that if $f$ is expansive then $f^n$ is also expansive for all $n\geq 1$ and 
$\gamma(f^n)\leq\gamma(f)$. However, it is unknown whether $\gamma(f^n)\to 0$ as $n\to+\infty$ or not, provided that the space has infinitely many points.
Again, this problem remains open only for subshifts since 
Sun \cite{Sun} proved that, in addition, assuming that the topological entropy is positive then $\gamma(f^n)$ decays exponentially to zero; and Fathi \cite{Fathi} proved that for spaces with positive topological dimension (\textit{i.e.}, not totally disconnected) the topological entropy is positive.

It is clear that the existence of doubly-asymptotic points implies that $\gamma(f^n)\to 0$ as $n\to\infty$. In this article we prove that a controlled decay of $\gamma(f^n)$ implies the existence of doubly-asymptotic points. For this purpose we consider hyperbolic metrics as introduced by Fathi \cite{Fathi}. In fact we use a self-similar hyperbolic metric which are known to exist for any expansive homeomorphism of a compact metric space \cite{Artigue}*{Theorem 2.3}.
We say that a compatible metric \(d\) on \(X\) is \textit{self-similar} if there exist constants \(c > 0\) and \(\lambda > 1\) such that if \(\dist (x,y) \leq c\) then \(\underset{\abs{i} = 1}{\max} \dist (f^{i}(x), f^{i}(y)) = \lambda \dist (x,y)\).

The purpose of this article is to connect these kind of concepts of topological dynamical systems with the nonstandard analysis developed by Robinson in the 1960s \cite{Robinson}. 
The application of nonstandard analysis techniques to study dynamical systems seems to have started in \cite{Hurd}. In this paper, besides making a detailed introduction of the nonstandard analysis concepts which are relevant for dynamical systems, Hurd translates and extends some notions concerning stability in the senses of Lagrange and Lyapunov, among other thing, into the nonstandard language. 
This program continued its development, see for instance \cite{Prazak}. 
In the present article we follow these lines of research with the focus in the previously mentioned problems. As usual, $\aster f$ denotes the nonstandard extension of the homeomorphism. We state our main result as follows.

\begin{theorem}
\label{MT}
An expansive dynamical system \((X,f)\) admits doubly-asymptotic points if and only if for a self-similar metric \(\dist\) with expanding factor \(\lambda\), there exists a standard real number \(C > 0\) and an infinite natural number \(N\) such that \(\gamma(\aster f^N) < \frac{C}{\lambda^{N/2}}\).
\end{theorem}

This result opens the problem of understanding the decay of $\gamma(f^n)$ for subshifts without doubly asymptotic points and with vanishing entropy; in particular to determine whether $\gamma(f^n)$ tends to zero or not.
For this question we provide a solution in the case of countable spaces. In Theorem \ref{thmCountable} we show that every expansive homeomorphism of a compact metric space with countably many points has doubly asymptotic points.
Thus, it essentially remains the case of minimal subshifts (in the sense that every orbit is dense).
In Example \ref{exampleIET} we provide the simplest example we know of such a minimal subshift; however we were not able to determine the behaviour of $\gamma(f^n)$ as $n\to\infty$, even in this case.

We also obtain a characterization of expansive homeomorphisms without doubly-asymptotic points in terms of what we call nonstandard expansivity. 
We say that $f$ is \textit{nonstandard expansive} if 
there exists a constant $c > 0$, such that for all $x \neq y$ there exists an infinite $n$ such that $\aster \dist (\aster f^n(x), \aster f^n(y)) > c$.
In Theorem \ref{comExpVer2} we show that
$f$ is nonstandard expansive if and only if $f$ is expansive and does not have asymptotic pairs.

This paper is organized as follows. 
In \S\ref{secPrelim} we introduce the definitions and preliminary results necessary for the rest of the paper. 
In \S\ref{secAppNSA} we show applications of nonstandard analysis for the simplification of some known proofs in expansive dynamical systems and for discovering new ones. 
In \S\ref{secNSE} we introduce nonstandard expansivity and prove a characterization in terms of doubly-asymptotic points.
In \S\ref{secExpCount} we show that countable (infinite) compact metric spaces do not admit nonstandard expansive systems.
Finally, in \S\ref{secHypMet} we prove Theorem \ref{MT}.

\section{Preliminaries}
\label{secPrelim}


After Robinson's original formulation \cite{Robinson}, several formalizations of nonstandard analysis emerged, such as the approach using superstructures, or Nelson's internal set theory \cite{Nelson}. The approach we will use in this paper is that of superstructures (Loeb and Wolf's book \cite{loebwolf} is an excellent introduction to this theory). In this approach, given a set \(X\), the following sequence of sets is constructed:
\begin{align*}
V_0(X) &= X, \\
V_{n+1}(X) &= V_n(X) \cup \mathcal{P}(V_n(X)).
\end{align*}
The set \(V(X) = \bigcup_{n \in \mathbb{N}} V_n(X)\) is called the \emph{superstructure} over \(X\). 
The elements in \(X\) are said to be of rank 0, and for all \(n \geq 1\), elements in \(V_n(X) \setminus V_{n-1}(X)\) are said to be of rank \(n\).
Every superstructure \(S = V(X)\) is associated with a language \(\mathcal{L}_X\), which is defined in the usual manner in mathematical logic, see \cite{loebwolf}*{p.~38}. While the semantics of the language can be formally defined, we will interpret the formulas in an intuitive way. We believe that this approach, which prioritizes clarity and accessibility, does not compromise the understanding of the proofs.

Given a superstructure \(S\) we can extend any mathematical structure 
\(\mathcal{X} = (X, R_1, \ldots, R_k, f_1, \ldots, f_k)\), where \(R_i\) are relations and \(f_j\) are functions, to the nonstandard extension 
$
\mathcal{\aster X} = \left(\aster X, \aster R_1, \ldots, \aster R_k, \aster f_1, \ldots, \aster f_k\right)
$
within the superstructure \(\aster S\)
constructed in \cite{loebwolf}*{p.~44}.
The fundamental tool of nonstandard analysis, which we will use in this paper, is the theorem known as the \emph{Transfer Principle}. This theorem is a consequence of the classical Łoś's theorem (see \cite{loebwolf}*{Theorem 2.5.6}), and it states that a formula \(\varphi\) holds in \(\mathcal{X}\) if and only if \(\aster \varphi\) holds in \(\mathcal{\aster X}\), where \(\aster \varphi\) is obtained by replacing each constant \(c\) in \(\varphi\) with \(\aster c\).
We consider \(\aster \mathbb{Z}\) and \(\aster \mathbb{R}\), the extensions of \(\mathbb{Z}\) and \(\mathbb{R}\) respectively, as included in the nonstandard extension of a superstructure. We denote \(\mathbb{Z}_{\infty}\) as the set of infinite integers.

Now we introduce some known relations between nonstandard analysis and metric spaces. 
Also, we derive some results concerning expansivity. Thoughout this paper $(X,\dist )$ will denote a metric space with nonstandard extension $(\aster X, \aster \dist )$. 
Let $\aster \R$ be the nonstandard extension of $\R$.
%
For $x, y \in \aster X$ we say that $x \sim y$ if $\aster \dist (x,y)$ is an infinitesimal.
%
If $r \in \aster \R$ is a bounded real number we define the \emph{standard part} of $r$ as $\st(r) \in \R$ such that $\st(r) \sim r$.
The next results will be used later.
%
\begin{theorem}[Robinson's Compactness Criterion, \cite{loebwolf}*{p. 86}]
\label{Robinson}
The metric space $X$ is compact if and only if for every $y \in {^*}X$, there exists $x \in X$ such that $x \sim y$.
\end{theorem}

\begin{proposition}[\cite{loebwolf}*{Theorem 1.9.2}]
\label{continuidadnoestandar}
If $x_0 \in X$, then $f\colon X\to X$ is continuous at $x_0$ if and only if for every $x \in {^*}X$ with $x \sim x_0$, we have ${^*}f(x) \sim f(x_0)$.
\end{proposition}


From now on $f\colon X \rightarrow X$ will denote a homeomorphism. 

Let us finally state some known facts about certain metric properties of expansive homeomorphisms.
We say that a metric $\dist$ is \textit{bi-Lipschitz} for $f$ with \textit{Lipschitz constant} $\lambda$ if for all $x, y \in X$ the following holds:
$$\max_{|i|=1} \dist (f^{i}(x), f^{i}(y)) \leq \lambda \dist (x,y).$$
We say that $\dist$ is a \textit{hyperbolic metric} for $f$ with \textit{expanding factor} $\lambda >1$ if for an expansivity constant $c$ it holds that if $\dist (x,y) \leq c$ then
$$\underset{\abs{i} =1}{\max} \dist(f^{i}(x), f^{i}(y)) \geq \lambda \dist (x,y).$$
In 
\cite{Fathi} Fathi proved that every expansive homeomorphism on a compact metric space admits a hyperbolic metric.
Given a homeomorphism $f\colon X \rightarrow X$, we say that a compatible hyperbolic metric $\dist$ on $X$ is \textit{self-similar} \cite{Artigue}
if there are constants $c > 0$, $\lambda > 1$ such that if $\dist (x,y) \leq c$, then
$$\underset{\abs{i} =1}{\max} \dist (f^{i}(x), f^{i}(y)) = \lambda \dist (x,y).$$
In \cite{Artigue}*{Theorem 2.3} it is shown that every expansive homeomorphism on a compact metric space admits a self-similar metric.
The expanding factor $\lambda$ can be taken as the bi-Lipschitz constant for $f$.

\section{Applications of Nonstandard Analysis}
\label{secAppNSA}
In this section we start applying nonstandard analysis to study expansivity.
First we show a nonstandard characterization of asymptotic points 
for a continuous map $f\colon X\to X$. 
Its proof is similar to \cite{Hurd}*{Theorem 3.1}.

\begin{proposition}
\label{proasint}
Two points $x,y\in X$ are asymptotic if and only if for every infinitely positive integer $m\in\aster\Z$ we have 
$\aster f^{m}(x) \sim \aster f^m(y).$
\end{proposition}

\begin{proof}
To prove the direct part suppose that 
$\lim\limits _{n \rightarrow + \infty} \dist (f^n(x), f^n(y))= 0$. Then, for any $\varepsilon \in \mathbb{R}$ with $\varepsilon > 0$, there exists $n_{\varepsilon} \in \mathbb{Z}^{+}$ such that the following formula holds:
$$(\forall m \in \mathbb{Z}^{+})(m \geq n_{\varepsilon} \rightarrow \dist (f^m(x),f^m(y)) < \varepsilon).$$
By the Transfer Principle we have that the following formula is also true:
$$(\forall m \in \aster \mathbb{Z}^{+})(m \geq n_{\varepsilon} \rightarrow \aster \dist (\aster f^m(x), \aster f^m(y)) < \varepsilon).$$
If $m$ is an infinitely positive integer, then $m > n_{\varepsilon}$ for any $\varepsilon$, thus $\aster \dist (\aster f^m(x), \aster f^m(y)) < \varepsilon$ holds for any $\varepsilon$, implying $\aster f^m(x) \sim \aster f^m(y)$.

To prove the converse suppose $\lim\limits _{n \rightarrow + \infty} \dist (f^n(x), f^n(y)) \neq 0$. 
Then, there exists $\varepsilon \in \mathbb{R}$ with $\varepsilon > 0$ 
such that for every $n \in \mathbb{N}$, the following holds:
$$(\exists m \in \mathbb{Z}^{+})((m \geq n) \wedge \dist (f^m(x),f^m(y)) > \varepsilon).$$
Then, taking a function $\psi : \mathbb{Z}^{+} \rightarrow \mathbb{Z}^{+}$ we have
$$(\forall n \in \mathbb{Z}^{+})(\psi(n) > n) \wedge (\dist (f^{\psi(n)}(x), f^{\psi(n)}(y)) > \varepsilon).$$
By the Transfer Principle we obtain:
$$(\forall n \in \aster \mathbb{Z}^{+}) (\aster \psi(n) > n) \wedge (\aster \dist (\aster f^{{^*}\psi(n)}(x), \aster f^{{^*}\psi(n)}(y)) > \varepsilon).$$
If $m$ is an infinitely positive integer then $\aster \psi(m)$ is also an infinitely positive integer. Since $\dist (f^{\psi(n)}(x), f^{\psi(n)}(y)) > \varepsilon$ it follows that $f^m(x) \not\sim f^m(y)$. 
\end{proof}

The next result characterizes asymptoticity for expansive homeomorphisms and is well known. 
We give a nonstandard proof.

\begin{lemma}\label{lemasint} 
Suppose that $f$ is an expansive homeomorphism.
If $x, y \in X$ and  
 there exists an expansivity constant $\delta$ such that for every $n \in \Z^{+}$, 
$\dist (f^n(x), f^n(y)) \leq \delta$, then $x$ and $y$ are asymptotic.
\end{lemma}

\begin{proof}
Applying the Transfer Principle to the formula
$(\forall n \in \Z^{+})(\dist (f^n(x), f^n(y)) \leq \delta)$ 
we have
$$(\forall n \in \aster \Z^{+})(\aster \dist (\aster f^n(x), \aster f^n(y)) \leq \delta).$$
Arguing by contradiction and applying Proposition \ref{proasint} we can 
take an infinitely positive integer $N$ such that $\aster f^N(x) \not \sim \aster f^N(y)$. 
Therefore, there exists a positive standard real number $r$ such that $r < \aster \dist (\aster f^N(x), \aster f^N(y))$.
By the compactness criterion of Robinson, there exist $x', y' \in X$ such that $\aster f^N(x) \sim x'$ and $\aster f^N(y) \sim y'$. 
By continuity, for every $n \in \Z$, we have $\aster f^{N+n}(x) \sim f^n(x')$ and $\aster f^{N+n}(y) \sim f^n(y')$. 
Therefore, $r < \dist (f^n(x), f^n(y)) \leq \delta$, which contradicts the expansivity hypothesis. 
\end{proof}

Suppose that $x,y\in \aster X$, $x\neq y$ and $x\sim y$. 
By continuity, for any standard integer $n\in\Z$ we have 
$\aster f^n(x)\sim\aster f^n(y)$. That is, $\aster\dist(\aster f^n(x),\aster f^n(y))$ is infinitesimal for all $n\in\Z$, even if $f$ is expansive. The following result shows that the expansiveness of the dynamics $(X, f)$ essentially extends to the dynamics $(\aster X, \aster f)$.

\begin{remark}
\label{rmkExpInfinitesimos}
Assuming expansiveness, for any two distinct points $x, y \in X$, there exists a time, minimal in absolute value, at which they are separated by a distance greater than the expansiveness constant. Consider the formula
\[
\varphi(x, y) := 
(x \neq y \rightarrow 
	(\exists n_{\min} \in \mathbb{Z} : \dist (f^{n_{\min}}(x), f^{n_{\min}}(y)) > c) \wedge (\forall n \in \mathbb{Z}, \dist (f^n(x), f^n(y)) > c \rightarrow \lvert n \rvert \geq |n_{\min}|)).
\]
By the Transfer Principle, the formula
    $
        (\forall x, y \in \aster X) \aster \varphi
    $
    also holds true. Therefore, the same result applies to any two points in $\aster X$ with respect to the dynamics $(\aster X, \aster f)$, using the same expansiveness constant as the dynamics $(X, f)$. In particular, if $(X,f)$ is an expansive dynamical system with expansiveness constant $c > 0$, then $(\aster X, \aster f)$ is also "expansive" in the sense that, for any distinct points $x,y$ in $\aster X$, there exists $n \in \aster \mathbb{Z}$ such that $\aster \dist(\aster f^n(x), \aster f^n(y)) > c$. Note that this implies that if $ x \sim y$, then there exists an infinite $n \in \aster \mathbb{Z}$ such that $\aster \dist(\aster f^n(x), \aster f^n(y)) > c$.
\end{remark}

We present below a nonstandard proof of a result due to Utz.
The graphical intuition shown in Figure \ref{fig:dibujo1} is the key.

\begin{theorem}[Utz \cite{utz1950unstable}]
\label{thmExistAsymptUtz}
If $f$ is expansive and $X$ has infinitely many points then there exist different asymptotic points $x,y \in X$ for $f$ or $f^{-1}$.
\end{theorem}

\begin{proof}
Since $X$ is infinite and compact there exists an accumulation point, which implies the existence of $x \in X$ and $y \in \aster X$ such that $x \sim y$. From Remark \ref{rmkExpInfinitesimos} there exists $m \in \aster \Z$, the minimum in absolute value, such that $ \aster \dist (\aster f^m(x), \aster f^m(y)) > c.$
Assuming that $m$ is positive we will show that there are asymptotic points for $f^{-1}$ (for $m$ negative the same argument gives asymptotic points for $f$). 
By continuity, for every $n \in \Z$ we have $\aster f^n(x) \sim \aster f^n(y)$, thus $m$ is infinite. Since $X$ is compact, by Theorem \ref{Robinson} there exist $x', y' \in X$ such that $x' \sim \aster f^m(x)$ and $y' \sim \aster f^m(y)$. See Figure \ref{fig:dibujo1}.

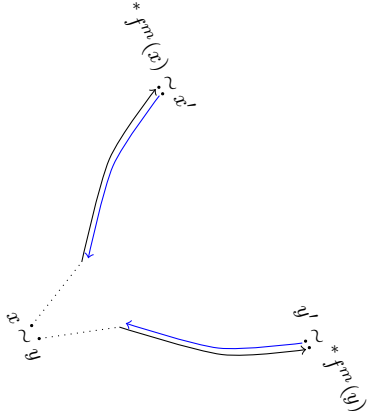
\begin{figure}[H]
  \centering
  \begin{tikzpicture}[rotate=30][scale=2][baseline={(0,0)}]
	\draw [draw=black,->] plot [smooth] coordinates {(1,0.5) (2,1.5) (3,2)};
	\draw[fill=black, draw=black] (3.05,2) circle (0.4pt);
	\draw[fill=black, draw=black] (3.05,1.9) circle (0.4pt);
	\node[below,font=\footnotesize,rotate=270] at (3.8,2) {\rotatebox{30}{$\aster f^m(x) \sim x'$}};
	\draw[fill=black, draw=black] (0,0.1) circle (0.4pt);
	\draw[fill=black, draw=black] (0,-0.1) circle (0.4pt);
	\draw [dotted] (0,0.1) -- (1,0.5);
	
	\node[below,font=\footnotesize,rotate=270] at (0.14,-0.16) {\rotatebox{30}{$x \sim y$}};
	\draw [draw=blue,->] plot [smooth] coordinates {(3,1.9) (2,1.4) (1.1,0.5)};
	\draw [draw=black,->] plot [smooth] coordinates {(1,-0.5) (2,-1.5) (3,-2)};
	\draw [dotted] (0,-0.1) -- (1,-0.5);
	\draw[fill=black, draw=black] (3.05,-2) circle (0.4pt);
	\draw[fill=black, draw=black] (3.05,-1.9) circle (0.4pt);
	\node[below,font=\footnotesize,rotate=270] at (3.8,-2.6) {\rotatebox{30}{$y'\sim \aster f^m(y)$}};
	\draw [draw=blue, ->] plot [smooth] coordinates {(3,-1.9) (2,-1.4) (1.1,-0.5)};
\end{tikzpicture}
  \caption{Construction of the asymptotic points $x'$ and $y'$.}
  \label{fig:dibujo1}
\end{figure}

By continuity, for every $n \in \Z^{+}$, we have $\aster f^{m-n}(x) \sim f^{-n}(x')$ and 
$\aster f^{m-n}(y) \sim f^{-n}(y')$, but 
$$\aster \dist (\aster f^{m-n}(x), \aster f^{m-n}(y)) \leq c$$ 
for every $n \in \Z^{+}$, therefore $\dist (f^{-n}(x'), f^{-n}(y')) \leq c$ for every $n \in \Z^{+}$. Hence, by Lemma \ref{lemasint}, we can conclude that $x'$ and $y'$ are asymptotic for $f^{-1}$.
\end{proof}

\begin{remark}
By \cite{Schw} we know that every expansive homeomorphism of a compact metric space with infinitely many points has asymptotic points in both senses, for $f$ and for $f^{-1}$. 
Some steps of this proof can be done by nonstandard analysis and are similar to those in the previous result.
\end{remark}

\section{Nonstandard expansivity}
\label{secNSE}
In this section we explore a variation of the definition of expansivity which seems natural from the nonstandard viewpoint. It is remarkable its relation with the existence of doubly asymptotic points.
We say that $f\colon X \rightarrow X$ is \textit{nonstandard expansive} if there exists a constant $c > 0$ such that for all $x,y\in X$, $x \neq y$, there exists an infinite $n$ such that $\aster \dist (\aster f^n(x), \aster f^n(y)) > c$.

\begin{theorem} \label{comExpVer2}
Let $f : X \rightarrow X$ be a homeomorphism. 
The following statements are equivalent:
\begin{enumerate}
\item $f$ is nonstandard expansive,
\item there is $c>0$ such that for all $x, y \in X$ with $x \neq y$ the set $\lbrace n \in \mathbb{Z} : d(f^n(x), f^n(y)) > c \rbrace$ is infinite,
\item $f$ is expansive 
without doubly-asymptotic points.
\end{enumerate}
\end{theorem}

\begin{proof}
($1\to 2$) 
Arguing by contradiction suppose that for all $c>0$ there are 
$x,y\in X$ such that 
the set $\lbrace n \in \mathbb{Z} : d(f^n(x), f^n(y)) > c \rbrace$ is finite.
In this case there exists $m \in \N$ such that 
\[
(\forall n \in \N)(n \geq m \rightarrow d(f^n(x),f^n(y)) \leq c).
\]
By the Transfer Principle, the following formula is also true:
\[
(\forall n \in \aster \N)(n \geq m \rightarrow \aster d(\aster f^n(x), \aster f^n(y)) \leq c).
\]
In particular, for all infinite positive integers $n$ we have 
$\aster d(\aster f^n(x), \aster f^n(y)) \leq c$. 
Similarly, $\aster d(\aster f^n(x), \aster f^n(y)) \leq c$ also holds for all infinite negative integers $n$.
Hence, $c$ is not an expansivity constant (of the nonstandard expansivity). 
As $c$ is an arbitrary positive real number we conclude that $f$ is not nonstandard expansive.

($2\to 3$) It is direct from the definitions.

($3\to 1$) 
Let $x, y \in X$. Since $x,y$ are not doubly-asymptotic there exists $m \in {^*}\mathbb{Z}_{\infty}$ such that $f^m(x) \not\sim f^m(y)$
and we can take $\alpha>0$ such that ${^*}\dist (f^m(x), f^m(y)) > \alpha$. By Robinson's compactness criterion, there exist $x', y' \in X$ such that $x' \sim f^m(x)$, $y' \sim f^m(y)$, and if ${^*}\dist (f^m(x), f^m(y)) > \alpha$ then $\dist (x', y') \geq \alpha$, in particular $x' \neq y'$. 
As $f$ is expansive there exists $n \in \mathbb{Z}$ such that $\dist (f^n(x'), f^n(y')) > c$. 
See Figure \ref{fig:diagram}.
\begin{figure}[H]
  \centering
  \begin{tikzpicture}[scale=1, baseline={(0,0)}]
    \definecolor{color817}{rgb}{0.1568627450980392,0.30196078431372547,0.6666666666666666}
    \definecolor{color1236}{rgb}{0.21176470588235294,0.2980392156862745,0.996078431372549}

    \foreach \i/\j in {0.41/-2.3292189, 1.37/-4.0292187, 1.55/-1.8692187, 2.73/-1.3092188, 2.65/-3.4692187, 
                       4.07/-2.6892188, 5.69/2.2307813, 7.57/-2.6692188, 6.01/1.6507813, 7.31/-1.9692187, 
                       9.13/2.9707813, 11.01/-0.40921876, 11.97/-1.3492187, 8.69/4.0307813} {
        \filldraw[black] (\i,\j) circle (1.2pt);
    }

    \foreach \i/\j in {3.19/-0.64921874, 3.47/-0.22921875, 3.83/0.15078124, 4.23/0.73078126, 4.51/1.1107812, 
                       4.91/1.3707813, 5.23/1.7507813, 5.43/1.9107813, 4.37/-2.7292187, 4.85/-2.7292187, 
                       5.25/-2.7292187, 5.63/-2.7292187, 6.07/-2.8292189, 6.51/-2.7892187, 6.89/-2.8292189, 
                       7.31/-2.7892187, 6.33/1.6907812, 6.71/1.8907813, 7.15/2.0507812, 7.53/2.2307813, 
                       7.99/2.4907813, 8.45/2.5907812, 8.83/2.8507812, 7.69/-1.9092188, 8.21/-1.7092187, 
                       8.75/-1.4892187, 9.31/-1.1692188, 10.03/-0.9092187, 10.45/-0.64921874, 10.73/-0.50921875, 
                       6.07/2.5307813, 6.49/2.7707813, 6.85/2.9707813, 7.23/3.2107813, 7.55/3.3507812, 
                       8.05/3.5307813, 8.47/3.7307813, 8.09/-2.6692188, 8.75/-2.4092188, 9.33/-2.3092186, 
                       10.07/-1.9892187, 10.45/-1.9892187, 11.01/-1.5692188, 11.57/-1.3892188} {
        \filldraw[black] (\i,\j) circle (0.6pt);
    }

    \node at (0.27453125,-2.0592186) {$x$};
    \node at (1.5745312,-4.459219) {$y$};
    \node at (1.4145312,-1.4992187) {$f(x)$};
    \node at (3.0345314,-3.7992187) {$f(y)$};
    \node at (2.4145312,-0.95921874) {$f^{2}(x)$};
    \node at (4.2345314,-3.1792188) {$f^{2}(y)$};
    \node at (5.3145313,2.5007813) {${^*}f^{m}(x)$};
    \node at (7.594531,-3.2592187) {${^*}f^{m}(y)$};
    \node at (6.134531,1.2407813) {$x'$};
    \node at (7.1945314,-1.6392188) {$y'$};
    \node at (9.234531,2.4207811) {$f^{n}(x')$};
    \node at (11.374531,-0.09921875) {$f^n(y')$};
    \node at (11.50,1.5607812) {\textcolor{color817}{greater than $c$}};
    \node at (12.80,-1.9392188) {$f^n({^*}f^{m}(y)) = {^*}f^{n +m}(y)$};
    \node at (8.474531,4.480781) {$f^n({^*}f^{m}(x)) = {^*}f^{n +m}(x)$};

    \node at (7.384531,-2.3992188) [rotate=-66] {$\sim$};
    \node at (5.744531,1.8207812) [rotate=-60] {$\sim$};
    \node at (8.844531,3.4007812) [rotate=-56] {$\sim$};
    \node at (11.424531,-0.93921876) [rotate=-52] {$\sim$};

    \draw[color1236, thick] (9.61,2.1507812) -- (10.49,0.31078124);
  \end{tikzpicture}
  \caption{Proving that $f$ is nonstandard expansive.}
  \label{fig:diagram}
\end{figure}
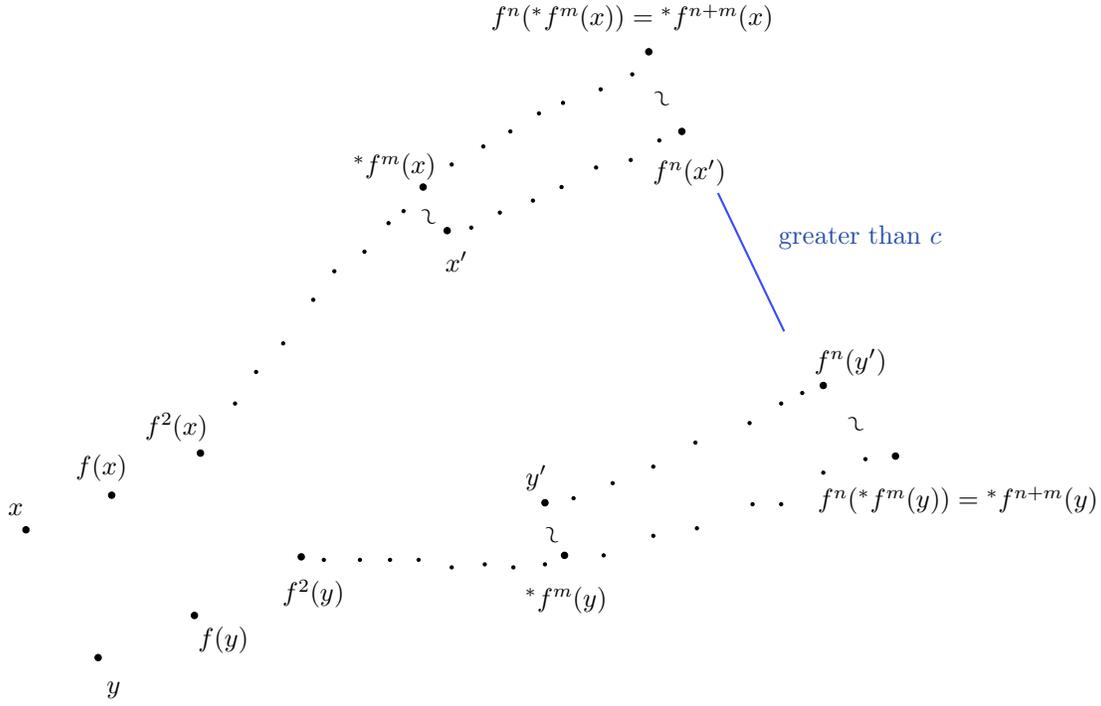
Since $f^n$ is continuous, ${^*}(f^n)(f^m(x)) \sim f^n(x')$, ${^*}(f^n)(f^m(y)) \sim f^n(y')$. Then, $\st({^*}\dist ({^*}f^n(f^m(x)), {^*}f^n(f^m(y)))) \geq c$, but this implies that $\dist (f^n(f^m(x)), f^n(f^m(y))) > \frac{c}{2}$. If $n \in \mathbb{Z}$ and $m \in {^*}\mathbb{Z}_{\infty}$, then $n + m \in {^*}\mathbb{Z}_{\infty}$, and thus $f$ is nonstandard expansive.
\end{proof}

The following subshift is an example of an expansive dynamics without asymptotic pairs; in particular, it is nonstandard expansive.

\begin{example}
\label{exampleIET}
Let $I = [0,1)$, $0 < a < b < 1$, and $a, b$ rationally independent. Let $T: I \rightarrow I$ be the function defined as follows:
\[T(x) =
\begin{cases}
x + 1 - a, & \text{si } x \in [0,a), \\
x - a + 1 - b, & \text{si } x \in [a,b), \\
x - b, & \text{si } x \in [b,1).
\end{cases}\]
For all $x$ in $I$, we define a sequence $(x_k)_{k \in \mathbb{Z}}$ as follows:
\[x_k =
\begin{cases}
0, & \text{si } T^k(x) \in [0,a), \\
1, & \text{si } T^k(x) \in [a,b), \\
2, & \text{si } T^k(x) \in [b,1).
\end{cases}\]
That is, for each $x$ in $I$, we define its \textit{itinerary} as $\iti(x) = (x_k)_{k \in \mathbb{Z}}$.

\[
\xymatrix{
\Sigma^{\mathbb{Z}} \ar[r]^{\sigma} & \Sigma^{\mathbb{Z}} \\
I \ar[u]^{\iti} \ar[r]_T & I \ar[u]_{\iti}
}
\]

Let $\Sigma = \{0, 1, 2\}$, $\sigma : \Sigma^{\mathbb{Z}} \rightarrow \Sigma^{\mathbb{Z}}$ be the shift, and $X = \overline{\bigcup_{x \in I} \iti(x)}$. Then, $\sigma\colon X \rightarrow X$ is expansive and has no doubly-asymptotic points. The details can be found in \cite{Yo}.
\end{example}

\section{Expansivity on countable spaces}
\label{secExpCount}
In this section we consider expansivity on a compact and countable (infinite) metric space $X$. 
For this kind of space we have some particular tools to use.
For every ordinal $\alpha$, we define $X^{(\alpha)}$ as the \textit{Cantor-Bendixson derivative} of $X$ by transfinite induction:
\begin{itemize}
\item $X^{(0)} = X$,
\item if $\alpha = \beta + 1$, then $X^{(\beta + 1)} = (X^{(\beta)})'$, the subset of accumulation points of $X^{(\beta)}$,
\item if $\alpha$ is an infinite limit ordinal, then $X^{(\alpha)} = \bigcap_{\beta < \alpha} X^{(\beta)}$.
\end{itemize}

If there exists an ordinal $\alpha$ such that $X^{(\alpha)}$ is finite
 we say that $\alpha$ is the derived degree of $X$, $\deg(X)=\alpha$. 
In \cite{Kato} it is shown that a countable and compact metric space admits an expansive homeomorphism if and only if $\deg(X)$ is not an infinite limit ordinal.

\begin{theorem} 
\label{thmCountable}
Let $(X,f)$ be an expansive dynamic. If $X$ is countable then $(X,f)$ has doubly-asymptotic points.
\end{theorem}

\begin{proof}
By \cite{Kato}*{Theorem 2.2} we know that $\deg(X) = \beta + 1$, so $X^{(\beta + 1)} = \{x_1, \ldots, x_n\}$. By transfinite induction, it is easy to prove that for every ordinal $\alpha$, $f(X^{(\alpha)}) = X^{(\alpha)}$.
This implies that the points $\{x_1, \ldots, x_n\}$ are periodic.
Notice that each of these points is fixed by $f^n$.
Therefore, we can restrict $f$ to $X^{(\beta)}$, and $f\vert_{X^{(\beta)}}: X^{(\beta)} \rightarrow X^{(\beta)}$ is an expansive homeomorphism. 
As $X^{(\beta)}$ has infinitely many points and is compact, 
we know from Theorem \ref{thmExistAsymptUtz}
that there are asymptotic points in $X^{(\beta)}$.
Thus, not all points in $X^{(\beta + 1)}$ are periodic. Therefore, there exists $x \in X^{(\beta)} \setminus X^{(\beta + 1)}$ such that $\alpha(x), \omega(x) \subset X^{(\beta + 1)}$. Then, for every positive infinite integer $N$ and for every negative infinite integer $M$, there exist $x_i$ and $x_j$ in $X^{(\beta + 1)}$ such that $\aster f^N(x) \sim x_i$ and $\aster f^M(x) \sim x_j$.

Let $y = f^n(x)$. We will prove that $x$ and $y$ are  asymptotic pairs.
Due to the continuity of $f$, if $\aster f^N(x) \sim x_i$, then $f^n(\aster f^N(x)) \sim f^n(x_i)$. Therefore, $\aster f^N(f^n(x)) \sim x_i$, which implies $\aster f^N(y) \sim x_i$. Hence, $\aster f^N(x) \sim \aster f^N(y)$. Thus, by Proposition \ref{proasint}, $x$ and $y$ are asymptotic. Similarly, $\aster f^M(x) \sim \aster f^M(y)$, and $x,y$ are asymptotic for $f^{-1}$.
\end{proof}

\section{Hyperbolic metrics}
\label{secHypMet}
In this section we will prove Theorem \ref{MT} stated in the Introduction of the article. 
The proof is given in two theorems.
Suppose that $f$ is expansive and recall
\[
\gamma(f)=\sup\{c \geq 0: c \text{ is an expansivity constant of } f\}.
\]

\begin{remark}
    If $(X, f)$ is an expansive dynamic with expansivity constant $c$, we know that $(\aster X, \aster f)$ is an expansive dynamic with the same expansivity constant $c$ in the sense of Remark \ref{rmkExpInfinitesimos}. It is an easy consequence of the Transfer Principle that $\gamma(\aster f)$ is the supremum of the expansivity constants of $\aster f$.
\end{remark}

\begin{lemma} \label{lemahiper}
Let $(X,\dist )$ be a compact metric space where $\dist$ is a hyperbolic metric with expanding factor $\lambda$, and $f\colon X \rightarrow X$ an expansive homeomorphism with expansivity constant $c$. Then, if for some $x,y \in X$, $\dist (f^n(x), f^n(y)) \leq c$ for all $n \geq 0$ (or for all $n \leq 0$), then
$$\dist (f^n(x), f^n(y)) \leq \frac{\dist (x,y)}{\lambda^{\abs{n}}}.$$
for all $n \geq 0$ (or for all $n \leq 0$).
\end{lemma}

\begin{proof}
We will prove the case $n \geq 0$; for $n \leq 0$ it is analogous. We see that for all $n \geq 0$,
$$\dist (f^{n+1}(x), f^{n+1}(y)) \leq \dfrac{\dist (f^n(x),f^n(y))}{\lambda}.$$
Suppose not, that is, there exists $n$ such that $\lambda \dist (f^{n+1}(x), f^{n+1}(y)) > \dist (f^n(x),f^n(y))$. Since 
$$\max \{ \dist (f^n(x), f^n(y)), \dist (f^{n+2}(x), f^{n+2}(y)) \} \geq \lambda \dist (f^{n+1}(x), f^{n+1}(y)),$$
then $\dist (f^{n+2}(x), f^{n+2}(y)) \geq \lambda \dist (f^{n+1}(x), f^{n+1}(y)).$ By induction, we deduce that for all $j \geq 0$,
$$\dist (f^{n+j}(x), f^{n+j}(y)) \geq \lambda^j \dist (f^n(x), f^n(y)).$$
This contradicts the fact that $\dist (f^n(x), f^n(y)) \leq c$ for all $n \geq 0$.
\end{proof}

\begin{theorem} \label{Directo}
Let $(X,\dist )$ be an expansive dynamical system with 
$x\neq y$ doubly-asymptotic. If $\dist$ is a hyperbolic metric with expanding factor $\lambda$, then there exists $C > 0$ (standard real) such that for every infinite $N$,
$$\gamma(\aster f^N) < \frac{C}{\lambda^{N/2}}.$$
\end{theorem}

\begin{proof}
As $x,y$ are doubly-asymptotic we have that $\aster f^m(x) \sim \aster f^m(y)$ for all infinite integer $m$. Take infinite $k$ and $m$ such that $\abs{k} > m$. Consider $k > m$, $\aster \dist (\aster f^k(x), \aster f^k(y)) = \aster \dist (\aster f^{k-m}(\aster f^m(x)), \aster f^{k-m}(\aster f^m(y))),$ but by Lemma \ref{lemahiper}, $\aster \dist (\aster f^{k-m}(\aster f^m(x)), \aster f^{k-m}(\aster f^m(y))) \leq \frac{c}{\lambda^{k-m}}$. Similarly, if we consider $k < -m$, we have $\aster \dist (\aster f^k(x), \aster f^k(y)) = \aster \dist (\aster f^{k+m}(\aster f^m(x)), \aster f^{k+m}(\aster f^m(y))) \leq \frac{c}{\lambda^{-k-m}}$. In conclusion, if $\abs{k} > m$, we have 
$$\aster \dist (\aster f^k(x), \aster f^k(y)) \leq \frac{c \lambda^m}{\lambda^{\abs{k}}}.$$
Therefore, if we take $N = 2m$, then for all $i \in \aster \Z$, we have $\aster \dist (f^{Ni}(x), \aster f^{Ni}(y)) \leq \frac{c\lambda^m}{\lambda^{\abs{Ni}}} \leq \frac{c\lambda^m}{\lambda^{2m}} = \frac{c}{\lambda^{N/2}}$. Taking $C = c+1$, we have $\gamma(\aster f^N) < \frac{C}{\lambda^{N/2}}$.
\end{proof}

\begin{theorem} \label{Reciproco}
Let $(X,f)$ be an expansive dynamical system and $\dist$ a bi-Lipschitz metric for $f$, with Lipschitz constant $\lambda$. If there exists $C > 0$ (standard real) and infinite $N$ such that $$\gamma(\aster f^N) < \frac{C}{\lambda^{N/2}},$$
then $f$ has 
doubly-asymptotic points.
\end{theorem}

\begin{proof}
If $\gamma(\aster f^N) \leq \frac{C}{\lambda^{N/2}}$, then there exist $x \neq y$ in $\aster X$ such that for all $k \in \aster \Z$, $\aster \dist (\aster f^{Nk}(x), \aster f^{Nk}(y)) < \frac{C}{\lambda^{N/2}}$. Since $\aster f$ is expansive with expansivity constant $c$, there exists $m$ (taking the smallest) such that $\aster \dist (\aster f^m(x), \aster f^m(y)) > \frac{c}{2}$. As we saw in the proof of Utz's theorem, this situation gives us two asymptotic points, either for $f$ or $f^{-1}$. We know that the following formula holds:
\[
(\forall m \in \aster \Z )(\forall N \in \aster \Z^{+})( \exists k \in \aster \Z^{+}) (N(k-1) < \abs{m} \leq Nk).
\]
Without loss of generality, we can assume that $k = 1$ for both $m$ and $N$, with $m < N$. Then, we have the configuration given by Figure \ref{fig:dibujo2}.

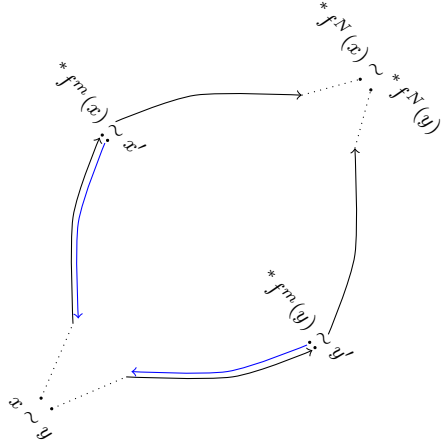
\begin{figure}[H]
  \centering
  \begin{tikzpicture}[rotate=45, scale=1]
	\draw [draw=black,->] plot [smooth] coordinates {(1,0.5) (2,1.5) (3,2)};
	\draw[fill=black] (3.05,2) circle (0.4pt);
	\draw[fill=black] (3.05,1.9) circle (0.4pt);
	\node[below,font=\footnotesize,rotate=270] at (3.8,1.7) {\rotatebox{45}{$\aster f^m(x) \sim x'$}};
	\draw[fill=black] (0,0.1) circle (0.4pt);
	\draw[fill=black] (0,-0.1) circle (0.4pt);
	\draw [dotted] (0,0.1) -- (1,0.5);
	
	\node[below,font=\footnotesize,rotate=270] at (0,-0.3) {\rotatebox{45}{$x \sim y$}};
	\draw [draw=blue,->] plot [smooth] coordinates {(3,1.9) (2,1.4) (1.1,0.5)};
	\draw [draw=black,->] plot [smooth] coordinates {(1,-0.5) (2,-1.5) (3,-2)};
	\draw [dotted] (0,-0.1) -- (1,-0.5);
	\draw[fill=black] (3.05,-2) circle (0.4pt);
	\draw[fill=black] (3.05,-1.9) circle (0.4pt);
	\node[below,font=\footnotesize,rotate=270] at (3.8,-2.2) {\rotatebox{45}{$\aster f^m(y) \sim y'$}};
	\draw [draw=blue, ->] plot [smooth] coordinates {(3,-1.9) (2,-1.4) (1.1,-0.5)};
	
	\draw [draw=black,->] plot [smooth] coordinates {(3.3,2)(4.3,1.5) (5.3,0.5)};
	\draw [dotted] (5.3,0.5) -- (6,0.1);
	\node[below,font=\footnotesize,rotate=270] at (7,-0.7) {\rotatebox{45}{$\aster f^N(x) \sim \aster f^N(y)$}};
	
	\draw[fill=black] (6,0.1) circle (0.4pt);
	\draw [draw=black,->] plot [smooth] coordinates {(3.3,-2)(4.3,-1.5) (5.3,-0.5)};
	\draw [dotted] (5.3,-0.5) -- (6,-0.1);
	\draw[fill=black] (6,-0.1) circle (0.4pt);
	
\end{tikzpicture}
\caption{Asymptotic points $x',y'$.}
  \label{fig:dibujo2}
\end{figure}

Let $m' \in \aster \Z$ be the smallest such that $\aster \dist (\aster f^{N-m'}(x), \aster f^{N-m'}(y)) > \frac{c}{2}$.
Suppose $m \leq m'$. Now let $x'$ and $y'$ in $X$ be such that $\aster f^m(x) \sim x'$ and $\aster f^m(y) \sim y'$. We have already seen that $x'$ and $y'$ are asymptotic. By the same argument, if $x''$ and $y''$ in $X$ satisfy $\aster f^{N-m'}(x) \sim x''$ and $\aster f^{N-m'}(y) \sim y''$, they will also be asymptotic. If we can prove that the iterates between $\aster f^m(x)$ and $\aster f^{N-m'}(x)$ are finite, i.e., there exists a finite $h$ such that $\aster f^{m+h}(x) = \aster f^{N-m'}(x)$, then by continuity we will have $\aster f^{m+h}(x') \sim \aster f^{N-m'}(x)$ and $\aster f^{m+h}(y') \sim \aster f^{N-m'}(y)$.
Therefore, $x'$ and $y'$ will be doubly-asymptotic.

Let us show that such finite $h$ exists. Indeed, since $f$ is bi-Lipschitz, the following formula holds
\[
(\forall x, y \in X )(\forall n \in \Z)( \dist (f^n(x), f^n(y)) \leq \lambda^{n} \dist (x,y))
\]
and by the Transfer Principle we conclude	
\[
(\forall x, y \in \aster X )(\forall n \in \aster \Z )(\aster \dist (\aster f^n(x), \aster f^n(y)) \leq \lambda^{n} \aster \dist (x,y)).
\]

\begin{figure}[H]
  \centering
  \begin{tikzpicture}[rotate=45][scale=1.5][baseline={(0,0)}]
	\draw [draw=black,->] plot [smooth] coordinates {(1,0.5) (2,1.5) (3,2)};
	\draw[fill=black, draw=black] (3.05,2) circle (0.4pt);
	\draw[fill=black, draw=black] (3.05,1.9) circle (0.4pt);
	\node[below,font=\footnotesize,rotate=270] at (3.8,1.7) {\rotatebox{45}{$\aster f^m(x) \sim x'$}};
	\draw[fill=black, draw=black] (0,0.1) circle (0.4pt);
	\draw[fill=black, draw=black] (0,-0.1) circle (0.4pt);
	\draw [dotted] (0,0.1) -- (1,0.5);
	
	\node[below,font=\footnotesize,rotate=270] at (0.1,-0.3) {\rotatebox{45}{$x \sim y$}};
	\draw [draw=blue,->] plot [smooth] coordinates {(3,1.9) (2,1.4) (1.1,0.5)};
	\draw [draw=black,->] plot [smooth] coordinates {(1,- 0.5) (2,- 1.5) (3,- 2)};
	\draw [dotted] (0,-0.1) -- (1,-0.5);
	\draw[fill=black, draw=black] (3.05,-2) circle (0.4pt);
	\draw[fill=black, draw=black] (3.05,-1.9) circle (0.4pt);
	\node[below,font=\footnotesize,rotate=270] at (3.8,-2.8) {\rotatebox{45}{$y' \sim \aster f^m(y)$}};
	\draw [draw=blue, ->] plot [smooth] coordinates {(3,-1.9) (2,- 1.4) (1.1,-0.5)};
	\draw [draw=black,->] plot [smooth] coordinates {(5,2) (6,1.5) (7,0.5)};
	\draw [draw=blue,->] plot [smooth] coordinates {(5,1.9) (6,1.4) (7,0.4)};
	\draw [dotted] (7,0.5) -- (8,0.1);
	\draw [draw=black,->] plot [smooth] coordinates {(5,-2) (6,-1.5) (7,-0.5)};
	\draw [dotted] (7,-0.5) -- (8,-0.1);
	\draw [draw=blue,->] plot [smooth] coordinates {(5,-1.9) (6,-1.4) (7,-0.4)};
	\draw[fill=black, draw=black] (8,0.1) circle (0.4pt);
	\draw[fill=black, draw=black] (8,-0.1) circle (0.4pt);
	\draw[fill=black, draw=black] (4.95,2) circle (0.4pt);
	\draw[fill=black, draw=black] (4.95,1.9) circle (0.4pt);
	\node[below,font=\footnotesize,rotate=270] at (5.7,1.4) {\rotatebox{45}{$\aster f^{N-m'}(x) \sim f^h(x')$}};
	\draw[fill=black, draw=black] (4.95,-2) circle (0.4pt);
	\draw[fill=black, draw=black] (4.95,-1.9) circle (0.4pt);
	\node[below,font=\footnotesize,rotate=270] at (5.7,-3.1) {\rotatebox{45}{$f^h(y') \sim \aster f^{N-m'}(y)$}};
	\node[below,font=\footnotesize,rotate=270] at (8.9,-0.7) {\rotatebox{45}{$\aster f^N(x) \sim \aster f^{N}(y)$}};
	\draw [dotted, black] (3,2) -- (5,2);
    \draw [dotted, black] (3,-2) -- (5,-2);
\end{tikzpicture}
  \caption{Proving doubly-asymptoticity.}
  \label{fig:dibujo3}
 \end{figure}
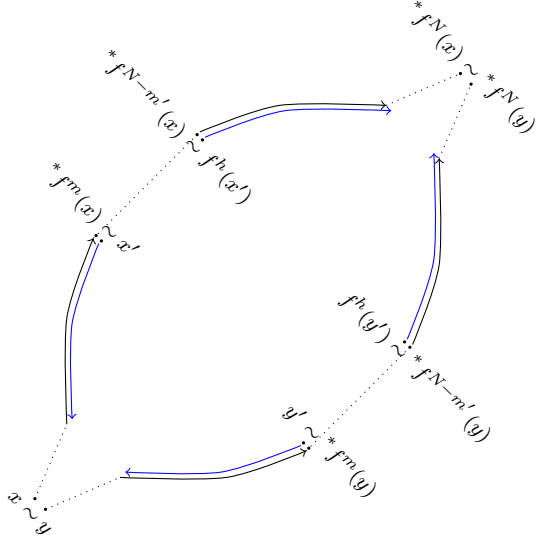
Therefore, we have:

\[
\frac{c}{2} < \aster \dist (\aster f^m(x), \aster f^m(y)) \leq \lambda^{m} \aster \dist (x,y) \leq \lambda^m \frac{C}{\lambda^{N/2}}
\]
and applying logarithms we obtain:
\[
\log\left(\dfrac{c}{2}\right) < (m  - \frac{N}{2}) \log \lambda + \log C.
\]
Therefore,
\[
N < 2m + 2\frac{\log(\frac{2C}{c})}{\log \lambda}.
\]
But $2\dfrac{\log(\frac{2C}{c})}{\log \lambda}$ is finite and $m \leq m'$, see Figure \ref{fig:dibujo3}. Therefore, there exists a finite $h$ such that $m + h = N - m'$.
\end{proof}

We can combine Theorems \ref{Directo} and \ref{Reciproco} to prove 
the Theorem \ref{MT} stated in \S\ref{secIntro}.


\begin{proof}[Proof of Theorem \ref{MT}]
If there are doubly-asymptotic points we can take a self-similar metric $\dist$, which in particular is hyperbolic with expanding factor $\lambda$, and by Theorem \ref{Directo}, we have
$$\gamma(\aster f^N) < \frac{C}{\lambda^{N/2}}.$$
Conversely, if we have the above inequality for a self-similar metric, which is in particular bi-Lipschitz, then by Theorem \ref{Reciproco} we ensure the existence of asymptotic pairs.
\end{proof}

\end{document}